\documentclass[10pt,reqno]{amsart}

\usepackage{pifont}
\usepackage{mathrsfs}
\usepackage{amscd,color}
\usepackage{amsmath}
\usepackage{latexsym}
\usepackage{amsfonts}
\usepackage{amssymb}
\usepackage{amsthm}
\usepackage{graphicx}
\usepackage{amsmath,amsfonts}
\usepackage{CJK}
\usepackage[linkcolor=blue,citecolor=blue]{hyperref}

\makeatletter
\newcommand{\Extend}[5]{\ext@arrow0099{\arrowfill@#1#2#3}{#4}{#5}}
\makeatother

\def\eps\epsilon

\def\cA{{\cal A}}

\def\C{\mathop{\bf C\kern 0pt}\nolimits}
\def\DD{\mathop{\bf D\kern 0pt}\nolimits}
\def\K{\mathop{\bf K\kern 0pt}\nolimits}
\def\N{\mathop{\bf N\kern 0pt}\nolimits}
\def\Q{\mathop{\bf Q\kern 0pt}\nolimits}

\newcommand{\beq}{\begin{equation}}
\newcommand{\eeq}{\end{equation}}
\newcommand{\ben}{\begin{eqnarray}}
\newcommand{\een}{\end{eqnarray}}
\newcommand{\beno}{\begin{eqnarray*}}
\newcommand{\eeno}{\end{eqnarray*}}

\def\R{\mathop{\mathbb R\kern 0pt}\nolimits}

\newtheorem{theorem}{Theorem}[section]
\newtheorem{proposition}[theorem]{Proposition}
\newtheorem{lemma}[theorem]{Lemma}

\newtheorem{corollary}[theorem]{Corollary}

\theoremstyle{remark}

\begin{document}

 \title[Inhomogeneous Biharmonic NLS]{ \bf Global Dynamics of the Non-Radial Energy-Critical Inhomogeneous Biharmonic NLS}

\author[C. M. Guzm\'an]{Carlos M. Guzm\'an}
\address{Department of Mathematics, UFF, Brazil}
\email{carlos.guz.j@gmail.com}
\author[S. Keraani]{Sahbi Keraani}
\address{Laboratoire Paul Painlevé UMR 8524, Université de Lille CNRS, 59655 Villeneuve d’Ascq Cedex, France}
\email{sahbi.keraani@univ-lille.fr}
\author[C. Xu]{Chengbin Xu}
\address{School of Mathematics and Statistics, Qinghai Normal University, Xining, Qinghai 810008, P.R.
China}
\email{xcbsph@163.com}

\begin{abstract}


We investigate the focusing inhomogeneous nonlinear biharmonic Schrödinger equation
\[
i\partial_t u + \Delta^2 u - |x|^{-b}|u|^p u = 0 \quad \text{on } \mathbb{R} \times \mathbb{R}^N,
\]
in the energy-critical regime, $p = \frac{8 - 2b}{N - 4}$, and $5 \leq N < 12$. We focus on the challenging non-radial setting and establish global well-posedness and scattering under the subcritical assumption $
\sup_{t \in I} \|\Delta u(t)\|_{L^2} < \|\Delta W\|_{L^2},
$
where $W$ denotes the ground state solution to the associated elliptic equation.

In contrast to previous results in the homogeneous case ($b = 0$), which often rely on radial symmetry and conserved quantities, our analysis is carried out without symmetry assumptions and under a non-conserved quantity, the kinetic energy. The presence of spatial inhomogeneity combined with the fourth-order dispersive operator introduces substantial analytical challenges. To overcome these difficulties, we develop a refined concentration-compactness and rigidity framework, based on the Kenig-Merle approach \cite{KM}, but more directly inspired by recent work of Murphy and the first author \cite{CM} in the second-order inhomogeneous setting.  

\


\end{abstract}
\maketitle
\setcounter{section}{0}\setcounter{equation}{0}
\section{Introduction}

\noindent

 We consider the initial value problem (IVP) for the focusing inhomogeneous nonlinear biharmonic Schrödinger equation (IBNLS):
\begin{align}\label{IBNLS}
\begin{cases}
i\partial_t u + \Delta^2 u = |x|^{-b} |u|^p u, & t \in \mathbb{R},\ x \in \mathbb{R}^N, \\
u(0,x) = u_0(x) \in \dot{H}^2(\mathbb{R}^N),
\end{cases}
\end{align}
where \( u: \mathbb{R} \times \mathbb{R}^N \to \mathbb{C} \), \( p = \frac{8 - 2b}{N - 4} \), and \( 0 < b < \min\left\{ \frac{8 - N}{2}, \frac{8}{N} \right\} \). The power-type nonlinearity is modulated by a singular spatial weight \( |x|^{-b} \), which breaks translation invariance and introduces significant analytic difficulties.

Equation \eqref{IBNLS} is invariant under the scaling
\[
u_\lambda(t,x) := \lambda^{\frac{N-4}{2}} u(\lambda^4 t, \lambda x),
\]
and satisfies the criticality condition
\[
\|u_\lambda(0)\|_{\dot{H}^2} = \|u_0\|_{\dot{H}^2},
\]
so it is said to be energy-critical or \( \dot{H}^2(\mathbb{R}^N) \)-critical. The equation formally conserves the energy
\[
E(u) := \frac{1}{2} \int_{\mathbb{R}^N} |\Delta u|^2\, dx - \frac{1}{p+2} \int_{\mathbb{R}^N} |x|^{-b} |u|^{p+2}\, dx,
\]
whenever the solution is sufficiently regular.
\medskip

In the homogeneous case ($b = 0$), equation \eqref{IBNLS} reduces to the classical biharmonic nonlinear Schrödinger equation (BNLS), which has received considerable attention as a higher-order analogue of the classical nonlinear Schrödinger equation (NLS). The BNLS features fourth-order dispersion, introducing significant analytical difficulties, such as the lack of Galilean invariance, the failure of virial identities, and the nonlocal nature of the propagator. In the energy-critical setting, global well-posedness and scattering for radial initial data in $\dot{H}^2(\mathbb{R}^N)$ with energy and kinetic norms below those of the ground state were independently established in \cite{Pausader09} and \cite{Miao-Xu-Zhao09}. In contrast, the corresponding scattering problem in the inhomogeneous case, where the spatial weight modulates the non-linearity $|x|^{-b}$, remains largely unexplored, with no known results available even under radial symmetry assumptions. In this work, we initiate the study of this problem and establish global well-posedness and scattering for general initial data in the energy space, without any symmetry assumptions. The presence of the spatial weight introduces new technical challenges, but also provides localization properties that play a key role in our analysis.

\medskip
In the second-order case, the energy-critical inhomogeneous nonlinear Schrödinger equation (INLS) has received significant attention. Murphy and the first author~\cite{GM} extended the concentration-compactness and rigidity method of Kenig–Merle~\cite{KM} to treat the non-radial INLS. Later, Guzmán and Xu~\cite{GXu2024} further developed this analysis in higher dimensions. Motivated by these results, we study the energy-critical inhomogeneous biharmonic nonlinear Schrödinger equation without assuming radial symmetry. The combination of the singular spatial weight and the fourth-order dispersion introduces substantial analytical difficulties, particularly in the absence of conserved quantities.

\medskip
It is worth mentioning that the Cauchy problem \eqref{IBNLS} was first studied by Guzmán–Pastor \cite{GP2020}, who established well-posedness in the subcritical case $\dot{H}^2$ and the intercritical regime, that is, $\frac{8 - 2b}{N} < p < \frac{8 - 2b}{N - 4}$. Improved results, extending the admissible ranges for both $p$ and $b$, were subsequently obtained by the same authors \cite{GP2022}. Liu and Zhang made further contributions \cite{LZ2021}, who employed Besov space techniques to refine the well-posedness theory. Scattering in the intercritical case was established by Campos–Guzmán \cite{GC2022} for $N \geq 5$, and more recently, the low-dimensional setting was addressed by Dinh–Keraani \cite{DK}. Both works adopted a streamlined approach introduced by Murphy \cite{Mu} for the second-order equation, avoiding the Kenig–Merle concentration-compactness and rigidity framework. The energy-critical case was also studied by Guzmán–Pastor \cite{GP2022}, who proved local well-posedness in $\dot{H}^2$ under certain restrictions on $b$. In this paper, we first revisit this well-posedness result, extending the admissible range for $b$.

\begin{proposition}\label{WP}
Assume $5\leq N \leq 11$, $p = \frac{8 - 2b}{N - 2}$, and $0 < b \leq \min\left\{\frac{8}{N-2}, \frac{12 - N}{2}\right\}$. Then the following holds:
\begin{itemize}
    \item[(i)] For any initial data $u_0 \in \dot{H}^2$, there exists $T = T(u_0) > 0$ and a unique solution $u$ to \eqref{IBNLS} with initial data $u(0) = u_0$, satisfying $ u \in L^q_{\mathrm{loc}}\big( (-T, T); \dot{H}^{2,r} \big)$, $\forall$ $(q,r)$ B-admissible.

    \item[(ii)] In particular, there exists $\delta > 0$ such that if\footnote{The norms used in this result are defined in Section $2$.}
    \[
    \|e^{it\Delta^2} u_0\|_{B([0, \infty))} + \|e^{it\Delta^2} u_0\|_{W([0, \infty))} < \delta,
    \]
    then the solution $u$ to \eqref{IBNLS} is global forward in time and satisfies $\|u\|_{B([0, \infty))} \leq 2\delta$.

 The analogous statement holds backward in the entire real line $\mathbb{R}$.
\end{itemize}
\end{proposition}

The conditions on the parameter \( b \) arise from the constraints \( p \geq b \) and \( p \geq 1 \); see Section~2.

We now state our main result: global well-posedness and scattering for the IVP \eqref{IBNLS} with non-radial initial data.

\begin{theorem}\label{main-T}
Let \(5 \leq N \leq 11\), \(0 < b <\min\{\frac{8}{N-2}, \frac{12 - N}{2}\}\) and \(p = \frac{8 - 2b}{N - 4}\). Suppose that \(u(t)\) is a solution to \eqref{IBNLS} with initial data \(u_0 \in \dot{H}^2(\mathbb{R}^N)\), defined on its maximal lifespan \(I\), and satisfies
\begin{equation}\label{kinetic-condition}
\sup_{t \in I} \|\Delta u(t)\|_{L^2} < \|\Delta W\|_{L^2},
\end{equation}
where $W$ denotes the ground-state solution to  \[
\Delta^2 W - |x|^{-b} |W|^p W = 0.
\]
Then \(u(t)\) exists globally in time and scatters both forward and backward in time in \(\dot{H}^2(\mathbb{R}^N)\).
\end{theorem}

\medskip

The proof of Theorem \ref{main-T} follows the general strategy pioneered by Kenig and Merle \cite{KM}, the so-called Kenig--Merle roadmap, with substantial adaptations to address the challenges posed by the non-radial setting and the lack of kinetic energy conservation. In addition to standard tools such as small-data global theory and perturbative stability, our approach hinges on two essential analytical ingredients. The first is a spatial localization result (Proposition \ref{N-profile}), which compensates for the absence of radial symmetry by ensuring that the initial data remains sufficiently separated from the origin. The second and most critical is a Palais--Smale-type condition (Proposition \ref{PS}), which enables the construction of a minimal blow-up solution, the so-called critical solution (Proposition \ref{CS}), with the lowest possible kinetic energy among all non-scattering solutions. The identification of this critical object is facilitated by the extraction of a bad profile and the decoupling of kinetic energy.

Adopting this framework, we proceed by contradiction and introduce the critical kinetic energy threshold
\[
K_c = \inf \left\{ K > 0 : L(K) = \infty \right\},
\]
where
\[
L(K) = \sup \left\{ \| u \|_{S(I)} : u \text{ solves \eqref{IBNLS} with } \sup_{t \in I} \| \Delta u(t) \|_{L^2} \leq K \right\}.
\]
The small-data theory ensures that $L(K) < \infty$ for sufficiently small $K$, and continuity of $L(K)$ follows from standard stability arguments. Thus, Theorem \ref{main-T} is equivalent to proving that
\[
K_c = \|\Delta W\|_{L^2},
\]
where $W$ is the ground state solution of the associated elliptic equation. Assuming, for contradiction, that $K_c < \|\Delta W\|_{L^2}$, we construct a sequence of blow-up solutions $\{u_n\}$ whose kinetic energies converge to $K_c$. A linear profile decomposition in $\dot{H}^2$ expresses $u_n(0)$ as a sum of orthogonal linear profiles plus a vanishing remainder. Each profile gives rise to a nonlinear evolution via Proposition \ref{N-profile}. The perturbation theory then yields the existence of a bad profile (Lemma \ref{bp}), a nonlinear profile responsible for the blow-up and lack of scattering. This profile leads to the construction of a global, non-scattering, and almost-periodic solution (in the sense of precompactness modulo scaling): the critical solution $u_c$ (Proposition \ref{CS}). Finally, a localized virial-type argument combined with conservation laws rules out the existence of such a solution. This contradiction completes the proof of Theorem \ref{main-T}, thereby establishing the scattering result for \eqref{IBNLS}.

As a consequence of our main result, we obtain a scattering criterion under sharp thresholds for the inhomogeneous biharmonic nonlinear Schr\"odinger equation. Specifically, if the initial data has energy and kinetic energy strictly below those of the ground state \( $W$ \), then the corresponding solution is global and scatters in \( \dot{H}^2 \).

\begin{corollary}\label{cor:scattering}
Let \( 5 \leq N \leq 11 \), \( 0 < b < \min\{\frac{8}{N-2} \frac{12 - N}{2}\} \), and $p = \frac{8 - 2b}{N - 4}$. Assume the initial data \( u_0 \in \dot{H}^2 \) satisfies
\begin{equation}\label{standard conditions}
 E[u_0] < E[W] \quad \text{and} \quad \| \Delta u_0 \|_{L^2} < \| \Delta W \|_{L^2},
\end{equation}
where \( $W$ \) is the ground state. Then the corresponding solution \( u \) to \eqref{IBNLS} exists globally in time and scatters in \( \dot{H}^2 \).
\end{corollary}

The conclusion follows by combining Lemma~\ref{energy-tra}, which provides uniform control of the \( \dot{H}^2 \)-norm strictly below \( \| W \|_{\dot{H}^2} \), with Theorem~\ref{main-T}, which ensures global existence and scattering.


We remark that the arguments developed in this paper for the focusing case can be readily adapted to the defocusing inhomogeneous biharmonic NLS. In this setting, the energy always controls the kinetic energy, and no threshold conditions are required. Consequently, one obtains the following global well-posedness and scattering result.

\begin{corollary}\label{cor:defocusing}
Let \( 5\leq N \leq 11 \) and $0 < b \leq \min\left\{ \frac{8}{N - 2}, \frac{12 - N}{2} \right\}$, and $p = \frac{8 - 2b}{N - 4}$.
Consider the defocusing inhomogeneous biharmonic NLS:
\[
i\partial_t u + \Delta^2 u + |x|^{-b} |u|^p u = 0, \quad u(0,x) = u_0 \in \dot{H}^2(\mathbb{R}^N).
\]
Then, for any initial data \( u_0 \in \dot{H}^2 \), there exists a unique global solution \( u  \), which scatters in \( \dot{H}^2 \).
\end{corollary}

We conclude the introduction by giving some notations which
will be used throughout this paper. We always use $X\lesssim Y$ to denote $X\leq CY$ for some constant $C>0$.
Similarly, $X\lesssim_{u} Y$ indicates there exists a constant $C:=C(u)$ depending on $u$ such that $X\leq C(u)Y$.
We also use the big-oh notation $\mathcal{O}$. e.g. $A=\mathcal{O}(B)$ indicates $A\leq CB$ for constant $C>0$.
The derivative operator $\nabla$ refers to the spatial  variable only.
We use $L^r(\mathbb{R}^N)$ to denote the Banach space of functions $f:\mathbb{R}^N\rightarrow\mathbb{C}$ whose norm
$$\|f\|_r:=\|f\|_{L^r}=\Big(\int_{\mathbb{R}^N}|f(x)|^r dx\Big)^{\frac1r}$$
is finite, with the usual modifications when $r=\infty$. For any non-negative integer $k$,
we denote by $H^{k,r}(\mathbb{R}^N)$ the Sobolev space defined as the closure of smooth compactly supported functions in the norm $\|f\|_{H^{k,r}}=\sum_{|\alpha|\leq k}\|\frac{\partial^{\alpha}f}{\partial x^{\alpha}}\|_r$, and we denote it by $H^k$ when $r=2$.
For a time slab $I$, we use $L_t^q(I;L_x^r(\mathbb{R}^N))$ to denote the space-time norm
\begin{align*}
  \|f\|_{L_{t}^qL^r_x(I\times \R^N)}=\bigg(\int_{I}\|f(t,x)\|_{L^r_x}^q dt\bigg)^\frac{1}{q}
\end{align*}
with the usual modifications when $q$ or $r$ is infinite, sometimes we use $\|f\|_{L^q(I;L^r)}$ or $\|f\|_{L^qL^r(I\times\mathbb{R}^N)}$ for short.

The remainder of this paper is structured as follows. In Section~2, we introduce the notation and recall the local well-posedness and stability theory for \eqref{IBNLS}, together with some variational properties associated with the ground state and the virial identity. In Section~3, we construct the minimal blow-up solution (Proposition~\ref{CS}), introducing the key ingredients required in our analysis. In Section~4, we complete the proof of the main result (Theorem~\ref{main-T}).

\section{Preliminaries}

\noindent

Let us start this section by introducing the notation used throughout the paper. We recall some Strichartz estimates associated to the linear biharmonic Schr\"odinger propagator.

 We define the set $\mathcal{B}_0$:
$$\mathcal{B}_0:=\left\{(q,r): \frac{4}{q}=\frac{N}2-\frac{N}r,\ 2\leq r<\frac{2N}{N-4}\right\}.$$


We also use the following estimates: if
 $F(x,z)=|x|^{-b}|z|^p z$,   then (see details in \cite[Remark 2.6]{Guzman} and \cite[Remark 2.5]{FG})
\begin{equation}\label{FEI}
 |F(x,z)-F(x,w)|\lesssim |x|^{-b}\left( |z|^p+ |w|^p \right)|z-w|
\end{equation}
and
\begin{equation}\label{SECONDEI}
\left|\nabla \left(F(x,z)-F(x,w)\right)\right|\lesssim  |x|^{-b-1}(|z|^{p}+|w|^{p})|z-w|+|x|^{-b}|z|^p|\nabla (z- w)|+E,
\end{equation}
where
\begin{eqnarray*}
 E &\lesssim& \left\{\begin{array}{cl}
 |x|^{-b}\left(|z|^{p-1}+|w|^{p-1}\right)|\nabla w||z-w| & \textnormal{if}\;\;\;p > 1 \vspace{0.2cm} \\
|x|^{-b}|\nabla w||z-w|^{p} & \textnormal{if}\;\;\;0<p\leq 1.
\end{array}\right.
\end{eqnarray*}

 Now, let us recall some results on Strichartz estimates.
\begin{lemma}
  Let $0\in I$, for any $(q,r),(\tilde{q},\tilde{r})\in \mathcal{B}_0$, the following statement hold

  (i)(linear estimate,\cite{Guo,Pau})
  $$\|e^{it\Delta^2}f\|_{L_t^qL_x^r(I\times\R^N)}\leq C\|f\|_{L^2};$$

   (ii)(nonlinear estimate,\cite{Guo,Pau})
  $$\left\|\int_{0}^{t}e^{i(t-s)\Delta^2}g(\cdot,s)ds\right\|_{L_t^qL_x^r(I\times\R^N)}\leq C\|g\|_{L_t^{\tilde{q}'}L_x^{\tilde{r}'}(I\times\R^N)};$$

  (iii) (nonlinear estimate,\cite{GP})
$$\left\|\Delta\int_{0}^{t}e^{i(t-s)\Delta^2}g(\cdot,s)ds\right\|_{L_t^qL_x^r(I\times\R^N)}\leq C\|\nabla g\|_{L_t^{2}L_x^{\frac{2N}{N+2}}(I\times\R^N)}.$$
\end{lemma}

\begin{lemma}[Local-smoothing, \cite{KPV91}]\label{L-smoothing}
Let $\varphi\in L^2(\R^N)$, then we have
\begin{align}
  \sup_{R>0}\frac1{R}\int_{\R}\int_{|x|\leq R}||\nabla|^{\frac32}e^{it\Delta^2}\varphi|^2dxdt\leq C_0\|\varphi\|_{L^2}.
\end{align}
\end{lemma}

\begin{lemma}\label{Local-smoothing}
  Given $f\in \dot{H}^2$, we have
  \begin{align}
    \|\nabla e^{it\Delta^2}f\|_{L_{t,x}^2([-T,T]\times\{|x|\leq R\})}^7\lesssim T^{\frac{20}{N+4}}R^{\frac{2N+48}{N+4}} \|e^{it\Delta^2}f\|_{B(\R)}^5\|\Delta f\|_{L_x^2}^2.
  \end{align}
\end{lemma}
\begin{proof}
  Let $K>0$, H\"older's and Bernstein's inequalities imply
 \begin{align*}
   \|\nabla e^{it\Delta^2}f_{\leq K}\|_{L_{t,x}^2([-T,T]\times\{|x|\leq R\})}\lesssim& T^{\frac8{N+4}}R^{\frac8{N+4}} \|\nabla e^{it\Delta^2}f_{\leq K}\|_{B(\R)}\\
   \lesssim&T^{\frac8{N+4}}R^{\frac8{N+4}}K\|e^{it\Delta^2}f\|_{B(\R)}.
 \end{align*}

 On the other hand, by local-smoothing (Lemma \ref{L-smoothing}), one can get
 \begin{align*}
   \|\nabla e^{it\Delta^2}f_{\geq N}\|_{L_{t,x}^2([-T,T]\times\{|x|\leq R\})}\lesssim& R \||\nabla|^{-\frac12}f_{\geq N}\|_{L_x^2}\\
   \lesssim& RK^{-\frac52}\|\Delta^2 f\|_{L_x^2}.
 \end{align*}
  The lemma follows by optimizing the choice of $K$.
\end{proof}


In what follows, we discuss some Cauchy problem results that will be essential to construct the critical solution. The well-posedness theory for \eqref{IBNLS} will be briefly revisited. To this end, we rely on a non-linear estimate that plays a crucial role in our analysis. In addition, we introduce the functional spaces that will be used throughout the paper.

\begin{align*}
   B(I):=&L_t^{\frac{2(N+4)}{N-4}}(I,L_x^{\frac{2(N+4)}{N-4}});\\
   W(I):=&L^{\frac{2(N+4)(b+1)}{b(N-2)+N-4}}(I,L^{\frac{2N(N+4)(b+1)}{N^2+b(N^2+8)+16}}).
\end{align*}

Let $N(I)$ denote by
$$\|F\|_{N(I)}\leq A$$
to indicate that there exists a decomposition $F=\sum_{j=1}^JF_j$ so that
$$\min\{\|\Delta F_j\|_{L_t^1L_x^2(I\times\R^N)},\|\nabla F_j\|_{L_t^2L_x^{\frac{2N}{N+2}}(I\times\R^N)}\}\leq \frac{A}J\ \ for\ all\ j.$$

The following nonlinear estimate will be used in the energy-critical regime.

\begin{lemma}[Nonlinear estimate \cite{GP21}]\label{Non-e1}
Let $0 < b \leq \min\left\{\frac{8}{N-2}, \frac{12 - N}{2}\right\}$. Then we have
   $$\||x|^{-b}|f|^{p} T g\|_{L_t^2L_x^{\frac{2N}{N+2}}}\lesssim \|\Delta f\|_{W(I)} \|f\|_{B(I)}^{p-b}\|\Delta g\|_{W(I)}^{b+1},$$
where $T\in \{\nabla,|x|^{-1}\}$.
\end{lemma}

We now establish the well-posedness result, stated in Proposition~\ref{WP}. This result improves the range for the parameter $b$ compared to the previous work of~\cite{GP}, where $b$ was restricted to a smaller interval.

\begin{proof}[\bf Proof of Proposition \ref{WP}] We only show (ii), the item (i) was proved in \cite{GP}. In order to improve the range for the parameter $b$, we replace the condition $\|e^{it\Delta^2} u_0\|_{B([0,\infty))}$ by
$\|e^{it\Delta^2} u_0\|_{B([0,\infty))} + \|\Delta e^{it\Delta^2} u_0\|_{W([0,\infty))}.$
Define
\[
S_\rho = \Big\{ u \in C(I; \dot{H}^2(\mathbb{R}^N)) : \|u\|_{B_2(I)} \leq \rho \Big\},
\]
where
\begin{equation}\label{norma2}
\|u\|_{B_2(I)} := \|u\|_{B(I)} + \|\Delta u\|_{W(I)}.
\end{equation}
We will next choose $\delta, \rho$ so that the operator $G$ defined by
\[
G(u)(t) = e^{it\Delta^2} u_0 + i \int_0^t e^{i(t-t')\Delta^2} \Big( |x|^{-b} |u(t')|^\alpha u(t') \Big) dt'
\]
is a contraction on $S_\rho$ equipped with the metric $d(u,v) := \|u-v\|_{B_2(I)}$.

Applying the Sobolev embedding and the Strichartz estimates, it follows that
\[
\|G(u)\|_{B(I)} \leq \|e^{it\Delta^2} u_0\|_{B(I)} + \left\| \Delta \int_0^t e^{i(t-s)\Delta^2} F(x,u)(s) ds \right\|_{L^{\frac{2(N+4)}{N-4}}L^{\frac{2N(N+4)}{N^2+16}}}
\]
\[
\leq \|e^{it\Delta^2} u_0\|_{S(I)} + c \|\nabla F(x,u)\|_{L^2_t L^{\frac{2N}{N+2}}_x (I)}.
\]
where the pair $\Big(\frac{2(N+4)}{N-4}, \frac{2N(N+4)}{N^2 + 16}\Big)$ is $\mathcal{B}_0$-admissible.

Moreover,
\[
\|\Delta G(u)\|_{W(I)} \leq \|\Delta e^{it\Delta^2} u_0\|_{W(I)} + C \|\nabla F(x,u)\|_{L^2_t L^{\frac{2N}{N+2}}_x (I)}.
\]
Since
\[
|\nabla F(x,u)| \leq |x|^{-b} |u|^\alpha |\nabla u| + |x|^{-b} |u|^\alpha |x|^{-1} |u|,
\]
and combining Lemma \ref{Non-e1}, we obtain
\[
\|\nabla F(x,u)\|_{L^2_t L^{\frac{2N}{N+2}}_x (I)} \leq c \|u\|_{B(I)}^{\alpha-b} \|\Delta u\|_{W(I)}^{b+1}.
\]
Thus,
\[
\|G(u)\|_{B(I)} \leq \|e^{it\Delta^2} u_0\|_{B(I)} + c \|u\|_{B(I)}^{\alpha-b} \|\nabla u\|_{W(I)}^{b+1},
\]
and
\[
\|\Delta G(u)\|_{W(I)} \leq \|\Delta e^{it\Delta^2} u_0\|_{W(I)} + c \|u\|_{B(I)}^{\alpha-b} \|\nabla u\|_{W(I)}^{b+1}. \tag{2.8}
\]
By summing the last two inequalities, we obtain
\[
\|G(u)\|_{B_2(I)} \leq \|e^{it\Delta^2} u_0\|_{B_2(I)} + c \|u\|_{B(I)}^{\alpha+1} < \delta + c \rho^{\alpha+1}.
\]
Choosing $\delta = \frac{\rho}{4}$ and requiring $C\rho^\alpha < \frac{1}{2}$, we ensure that
\[
\|G(u)\|_{B_2(I)} \leq \rho,
\]
which means $G(u) \in S_\rho$.

To complete the proof, we show that $G$ is a contraction on $S_\rho$. Repeating the computations above, we get
\[
d(G(u), G(v)) \leq C \|\nabla (F(x,u) - F(x,v))\|_{L^2_t L^{\frac{2N}{N+2}}_x (I)}. 
\]
Using \eqref{SECONDEI} and the fact that $\alpha \geq 1$, we have
\[
|\nabla (F(x,u) - F(x,v))| \lesssim |x|^{-b} (|u|^\alpha + |v|^\alpha) |x|^{-1} |u-v|
\]
\[
+ |x|^{-b} (|u|^\alpha + |v|^\alpha) |\nabla (u-v)| + E,
\]
where
\[
E \lesssim |x|^{-b} (|u|^{\alpha-1} + |v|^{\alpha-1}) |\nabla v| |u-v|.
\]
Hence, from Lemma \ref{Non-e1}, one has
\[
d(G(u), G(v)) \leq C \Big( \|\Delta u\|_{W(I)}^b \|u\|_{B(I)}^{\alpha-b} + \|\Delta v\|_{W(I)}^b \|v\|_{B(I)}^{\alpha-b} \Big) \|\Delta (u-v)\|_{W(I)} + E_1,
\]
where
\[
E_1 \lesssim
\begin{cases}
\Big( \|u\|_{B(I)}^{\alpha-1} + \|v\|_{B(I)}^{\alpha-1} \Big) \|\Delta v\|_{W(I)} \|\Delta(u-v)\|_{W(I)}^b \|u-v\|_{B(I)}^{1-b}, & \text{if } b < 1; \\
\|\Delta u\|_{W(I)}^{b-1} \|u\|_{B(I)}^{\alpha-b} \|\Delta v\|_{W(I)} \|\Delta(u-v)\|_{W(I)}, & \text{if } b \geq 1.
\end{cases}
\]
Therefore, if $u, v \in S_\rho$, we conclude that
\[
d(G(u), G(v)) \leq 4C \rho^\alpha d(u,v),
\]
which means that $G$ is a contraction (provided that $4C \rho^\alpha < 1$). Thus, by the contraction mapping principle, $G$ admits a unique fixed point $u \in S_\rho$. \qed

\medskip

Note that the conditions $b \leq \frac{12-N}{2}$ and $b \leq \frac{8}{N-2}$ in Proposition \ref{WP} arise from the constraints $\alpha \geq 1$ and $\alpha - b \geq 0$, respectively.
\end{proof}

Similarly as before, using Strichartz estimates and Lemma \ref{Non-e1}, we also obtain

\begin{corollary}
\begin{itemize}
    \item For any $\psi \in \dot{H}^2$, there exists $T > 0$ and a solution $u : (T, \infty) \times \mathbb{R}^N \to \mathbb{C}$ to \eqref{IBNLS} such that
    \[
    e^{-it\Delta^2} u(t) \to \psi \quad \text{in } \dot{H}^2 \text{ as } t \to \infty.
    \]
    The analogous statement holds backward in time.
    \item Scattering criterion: if the norms $
    \|u\|_{L^\infty_t \dot{H}^2_x}$ and  $\|u\|_{S([0, \infty))}$ are finite, then the solution $u$ scatters forward in time.
\end{itemize}
\end{corollary}

A key step in constructing the critical solution is a stability theory for the Cauchy problem, which allows comparison between approximate and exact solutions under smallness assumptions. The next proposition provides a stability property central to the concentration-compactness argument used to prove the existence of a critical solution.
\begin{proposition}[Long-time perturbation \cite{GP21}]\label{stability}
  Let $0<b\leq \min\left\{\frac{8}{N-2}, \frac{12 - N}{2}\right\}$. Suppose $\tilde{u}:I\times\R^N\rightarrow\mathbb{C}$ obeys
  $$\|\tilde{u}\|_{L_t^\infty\dot{H}_x^2}+\|\tilde{u}\|_{B(I)}\leq E<\infty.$$
  If there exists $\epsilon_1=\epsilon_1(E)>0$ such that
  \begin{align*}
    &\|(i\partial_t+\Delta^2)\tilde{u}+|x|^{-b}|\tilde{u}|^p\tilde{u}
    \|_{N(I)}\leq \epsilon<\epsilon_1,\\
    &\|e^{i(t-t_0)\Delta^2}[u_0-\tilde{u}(t_0)]\|_{B_2(I)}\leq \epsilon<\epsilon_1,
  \end{align*}
  for\footnote{Recall the norm $\|\cdot\|_{B_2}$ is given in \eqref{norma2}.} some $t_0\in I$ and $u_0\in\dot{H}^2$, then there exists a unique solution $u:I\times\R^N\rightarrow\mathbb{C}$ with $u(t_0)=u_0$, which satisfies
  $$\|u-\tilde{u}\|_{B(I)}\lesssim \epsilon$$
\end{proposition}

We now turn to the variational analysis of the equation, focusing on the sharp Sobolev inequality, energy trapping, and coercivity properties.
\begin{lemma}\label{profiles}
  Let $f_n$ be a bounded sequence in $\dot{H}^2(\R^N)$. Then there exist $J^*\in \mathbb{N}\cup\{\infty\}$; profiles $\phi^j\in \dot{H}^2(\R^N)\setminus\{0\}$; scales $\lambda_n^j\in(0,\infty)$; space translation parameters $x_n^j\in\R^N$ such that writing
  $$f_n(x)=\sum_{j=1}^J(\lambda_n^j)^{-\frac{N-4}2}\phi^j\Big(\frac{x-x_n^j}{\lambda_n^j}\Big)+r_n^J(x), \ \ 1\leq J\leq J^*.$$
  The decomposition has the following properties (it holds up to a subsequence)
  $$\limsup_{J\to J^*}\limsup_{n\to\infty}\|r_n^J\|_{L_x^{\frac{2N}{N-4}}}=0,$$
  $$\limsup_{n\to\infty}\Big|\|f_n\|_{\dot H^2}^2-\sum_{j=1}^J\|\phi^j\|_{\dot H^2}^2-\|r_n^J\|_{\dot H^2}^2\Big|=0,$$
  $$\liminf_{n\to\infty}\Big[\frac{|x_n^j-x_n^k|^2}{\lambda_n^j\lambda_n^k}
  +\log\frac{\lambda_n^k}{\lambda_n^j}\Big]=\infty,\ \forall j\neq k,$$
  $$(\lambda_n^j)^{\frac12}r_n^J(\lambda_n^j+x_n^j)\to0,\ weakly\ in\ \dot H^2,$$
  \begin{align}\label{jianjin-P}
   \limsup_{n\to\infty}\liminf_{n\to\infty}\big[P(f_n)-P(g_n^j\phi^j)-P(r_n^J)\big]=0.
  \end{align}
\end{lemma}
\begin{proof} Following the argument in \cite{Gerard1998}, it suffices to show \eqref{jianjin-P}. In fact, the proof of \eqref{jianjin-P} comes from \eqref{energy-decoupling} in Proposition \ref{Linear-Profile}.
\end{proof}
\begin{proposition}[Sharp Sobolev inequality]\label{SSI}
  Let $N\geq 5$ and $0<b<\min\{4,\frac{N}2\}$. There exists a positive constant $C_0$ such that for any $u\in\dot{H}^2(\R^N)$,
  $$\int_{\R^N}|x|^{-b}|u|^{p+2}dx\leq C_0\|u\|_{\dot{H}^2}^{p+2}.$$
  Moreover, there exists a function $W\in\dot{H}^2$ so that
  $$\int_{\R^N}|x|^{-b}|W|^{p+2}dx=C_0\|W\|_{\dot{H}^2}^{p+2},$$
  and $W$(ground state) is the solution of elliptic equation:
          \begin{align}\label{elliptic}
\Delta^2 W-|x|^{-b}|W|^{p}W=0.
          \end{align}
\end{proposition}
\begin{proof}
First, we prove the sharp Sobolev inequality. Using Sobolev embedding and Hardy inequality, we have
  \begin{align*}
   \int_{\R^N}|x|^{-b}|u|^{p+2}dx\leq& C\||x|^{-2}u\|_{L_x^2}^{\frac{b}2}\|u\|_{L_x^{\frac{2N}{N-4}}}^{p+2-\frac{b}{2}}\\
   \leq& C\|u\|_{\dot{H}^2}^{p+2}.
  \end{align*}

 Next, we turn to prove the sharpness of Sobolev inequality. For any $f\in \dot H^2(\R^N)$, let $P(f):=\int_{\R^N}|x|^{-b}|f|^{p+2}dx$. Define the functional $J(f)$ in $\dot H^2(\R^N)$:
  \begin{align}
    J(f)=\frac{P(f)}{\|\Delta f\|_{L^2}^{p+2}}.
  \end{align}
 By the invariant of $J(f)$ under the scaling $f_\lambda(x)=:f(\lambda x)$, there exists a sequence $\{f_n\}$ satisfying $\|f_n\|_{\dot H^2}=1$ and
  \begin{align*}
    \lim_{n\to\infty}J(f_n)=\sup_{f\in \dot{H}^2\setminus\{0\}}J(f)=C_0.
  \end{align*}

Due to Lemma \ref{profiles}, we obtain
  $$f_n=\sum_{j=1}^{M}g_n^j[\phi^j]+r_n^M.$$
 We claim that there exists one profile, that is  $\phi^j\equiv0$ for $j\geq2$.
In fact, by the asymptotic vanishing of potential in Lemma \ref{profiles}, we have
 \begin{align*}
   C_0\leq \lim_{n\to\infty}\sum_{j=1}^{\infty}P(g_n^j[\phi^j])\leq C_0\sum_{j=1}^{\infty}\|\Delta \phi^j\|_{L^2}^{p+2}.
 \end{align*}
On the other hand, we have the fact $\sum_{j=1}^{\infty}\|\Delta \phi^j\|_{L^2}^{2}\leq 1.$
 If there are more than one profile $\phi^j\neq0$, since $p+2>1$ we get
 $\sum_{j=1}^{\infty}\|\Delta \phi^j\|_{L^2}^{p+2}<1, $
 thus $C_0<C_0$, this is contradiction. Hence, there exists only one profile and $\|\Delta \phi^1\|_{L^2}=1$.

Since
 $(g_n^1)^{-1}[f_n](x)\rightharpoonup \phi^1$ weakly in
$\dot H^2$ and $\|f_{n}\|_{\dot H^2}=\|\phi^1\|_{\dot H^2}$, we have
 $$\lim_{n\to\infty}\|(g_n^1)^{-1}[f_n]-\phi^1\|_{\dot H^2}=0.$$
 Then we get $C_0=J(\phi^1)$.

Finally, let $\lambda=\|\phi^1\|_{\dot H^2}^{\frac{N-4}{2}}$, we define
 $$W(x)=\phi^1(\lambda^{-1}x).$$
  We can prove $W$ solve the elliptic equation \eqref{elliptic} via the variational derivatives of $J(f)$ (See \cite{Ta,We} for more details, we omit the proof). Then we complete the proof this statement.
 \end{proof}
\begin{lemma}[Energy trapping]\label{energy-tra}
Suppose that $u: I\times\R^N\to\mathbb{C}$ is a solution to \eqref{IBNLS} obeying \eqref{standard conditions}.
Then there exists $\delta>0$ such that
$$\sup_{t\in I}\|u(t)\|_{\dot{H}^2}<(1-\delta)\|W\|_{\dot{H}^2}.$$
  Furthermore,
  \begin{align}
    E(u(t))\sim\|u(t)\|_{\dot{H}^2}^2\sim\|u_0\|_{\dot{H}^2}^2\ \ for\ \ all\ \ t\in I.
  \end{align}
\end{lemma}
\begin{proof}
   Define the convex function 
   $$f(y)=\frac{y}2-\frac{C_0}{p+2}y^{\frac{p+2}2}.$$
   Thus $f$ is strictly increasing on $[0,C_0^{-\frac2p}]$ and strictly on $[C_0^{-\frac2p},\infty)$. By Proposition \ref{SSI}, one can get
   $$f(\|\Delta u_0\|_{L^2}^2)=\frac12\|\Delta u_0\|_{L^2}^2-\frac{C_0}{p+2}\|\Delta u_0\|_{L^2}^{p+2}\leq E(u_0)$$
   and $$f(\|\Delta W\|_{L^2}^2)=E(W)=\frac{p}{2(p+2)}C_0^{-\frac2p}.$$
   Thus $f^{-1}$ exists on $[0, \frac{p}{2(p+2)}C_0^{-\frac2p}]$ and is strictly increasing. By \eqref{standard conditions}, one has $E(u_0)\leq (1-\delta)E(W)$. Hence
   $$\|\Delta u(t)\|_{L^2}^2\leq f^{-1}(E(u_0))\leq f^{-1}((1-\delta)E(W))\leq (1-\delta_1)\|\Delta W\|_{L^2}^2.$$
   
   With above in hand, we have 
   \begin{align*}
     \frac12\|\Delta u(t)\|_{L^2}^2\geq E(u(t))&=\frac12\|\Delta u(t)\|_{L^2}^2-\frac{1}{p+2}\int_{\R^N}|x|^{-b}|u(t)|^{p+2}dx\\
     &\geq \frac12\|\Delta u(t)\|_{L^2}^2-\frac{C_0}{p+2}\|\Delta u(t)\|_{L^2}^{p+2}\\
     &\geq \|\Delta u(t)\|_{L^2}^2\Big(\frac12-\frac{1}{p+2}\Big(\frac{\|\Delta u(t)\|_{L^2}}{\|\Delta W\|_{L^2}}\Big)^{p}\Big)\\
     &\geq\Big(\frac12-\frac{(1-\delta_1)^p}{p+2}\Big)  \|\Delta u(t)\|_{L^2}^2,
   \end{align*}
which implies the lemma.
\end{proof}
\begin{lemma}[Coercivity]\label{L:coercive} Suppose that $u:I\times\R^N\to\C$ is a solution to \eqref{IBNLS}. Suppose $\sup_{t\in I}\|\Delta u(t)\|_{L^2}\leq (1-\delta_0)\|\Delta W\|_{L^2}$ for some $\delta_0>0$, then there exists $\delta>0$ such that
\[
\int |\Delta u(t,x)|^2 - |x|^{-b}|u(t,x)|^{p+2}\,dx \geq \delta \int |\Delta u(t,x)|^2\,dx \ \ \text{uniformly over t}\in I.
\]
\end{lemma}

\section{Existence of minimal non-scattering solution}\label{S:exist}

In this section, we establish several results that are essential for constructing a minimal non-scattering (or critical) solution, as formulated in Proposition \ref{CS} (below). Specifically, our argument relies on two central ingredients: Proposition \ref{N-profile}, which proves scattering solutions to \eqref{IBNLS} for initial data located sufficiently far from the origin and Proposition \ref{PS} (Palais-Smale Condition). As an initial step, we introduce the linear profile decomposition.


\begin{proposition}[Linear profile decomposition]\label{Linear-Profile}
  Let $u_n$ be a bounded sequence in $\dot{H}^2(\R^N)$. Then the following holds up to a subsequence:

  There exist $J^*\in \mathbb{N}\cup\{\infty\}$; profiles $\phi^j\in \dot{H}^2\setminus\{0\}$; scales $\lambda_n^j\in(0,\infty)$; space translation parameters $x_n^j\in\R^N$; time translation parameters $t_n^j$; and remainders $w_n^J$ so that writing
  $$g_n^jf(x)=(\lambda_n^j)^{-\frac{N-4}2}f(\frac{x-x_n^j}{\lambda_n^j}),$$
  we have the following decomposition for $1\leq j\leq J^*$:
  $$u_n=\sum_{j=1}^{J}g_n^j[e^{it_n^j\Delta^2}\phi^j]+w_n^J.$$
  This decomposition satisfies the following conditions:
  \begin{itemize}
    \item Energy decoupling: writing $P(u)=\||x|^{-b}|u|^{p+2}\|_{L^1}$, we have
    \begin{align}\label{energy-decoupling}
      &\lim_{n\to\infty}\{\|\Delta u_n\|_{L^2}^2-\sum_{j=1}^J\|\Delta \phi^j\|_{L^2}^2-\|\Delta w_n^J\|_{L^2}^2\}=0\\\label{potential-decoupling}
      &\lim_{n\to\infty}\{P(u_n)-\sum_{j=1}^JP(g_n^j[e^{it_n^j\Delta}\phi^j])-P(w_n^J)\}=0.
    \end{align}
    \item Asymptotic vanishing of remainders:
    \begin{align}\label{vanishing}
      \limsup_{J\to J^*}\limsup_{n\to\infty}\|e^{it\Delta^2}w_n^J\|_{B_0(\R)}=0.
    \end{align}
    \item Asymptotic orthogonality of parameters: for $j\neq k$,
    \begin{align}\label{orthogonality}
      \lim_{n\to\infty}\left\{\log\Big[\frac{\lambda_n^j}{\lambda_n^k}\Big]+\frac{|x_n^j-x_n^k|^2}{\lambda_n^j\lambda_n^k}
      +\frac{|t_n^j(\lambda_n^j)^4-t_n^k(\lambda_n^k)^4|}{\lambda_n^j\lambda_n^k}\right\}=\infty.
    \end{align}
  \end{itemize}
  In addition, we may assume that either $t_n^j\equiv0$ or $t_n^j\to\pm\infty$, and that either $x_n^j\equiv0$ or $|x_n^j|\to +\infty$.
\end{proposition}
\begin{proof}
  We only show the energy decoupling \eqref{energy-decoupling}, the other case can be showed by same way as in \cite{Ker}. In fact, following the argument of Lemma 4.1 in \cite{CHL}, it is easily to get the energy decoupling.
\end{proof}

The following proposition serves as the first essential ingredient in the construction of the critical solution. It establishes the scattering of solutions to~\eqref{IBNLS} with initial data that are sufficiently localized away from the origin. This result plays a crucial role in extending the construction of a minimal blow-up solution from the radial to the non-radial setting. Furthermore, it guarantees that any compact solution must remain spatially localized near the origin, which is essential for applying the localized virial argument later in the proof.

\begin{proposition}[Nonlinear Profile Away from the Origin]\label{N-profile}
Let $\lambda_n \in (0, \infty)$, $x_n \in \mathbb{R}^N$, and $t_n \in \mathbb{R}$ satisfy
\[
\lim_{n \to \infty} \frac{|x_n|}{\lambda_n} = \infty, \quad t_n \equiv 0 \text{ or } t_n \to \pm \infty.
\]
Let $\phi \in \dot{H}^2$, and define
\[
\phi_n(x) = g_n\big[ e^{i t_n \Delta^2} \phi \big](x) = \lambda_n^{-\frac{N-4}{2}} \big(e^{i t_n \Delta^2} \phi\big)\left( \frac{x - x_n}{\lambda_n} \right).
\]
Then for all sufficiently large $n$, there exists a global solution $v_n$ to \eqref{IBNLS} satisfying
\[
v_n(0) = \phi_n, \quad \|v_n\|_{B(\mathbb{R})} \lesssim 1,
\]
with implicit constant depending only on $\|\phi\|_{\dot{H}^2}$.

Moreover, for any $\varepsilon > 0$, there exist $K \in \mathbb{N}$ and $\psi \in C_c^\infty(\mathbb{R} \times \mathbb{R}^N)$ such that for all $n \geq K$,
\[
\| v_n - g_n[\psi] \|_{B(\mathbb{R})} < \varepsilon.
\]
\end{proposition}

\begin{proof}
Let $\theta \in (0, 1)$ be a small parameter to be chosen later. Introduce the frequency cutoff operator
\[
P_n = P_{ \left| \frac{x_n}{\lambda_n} \right|^{-\theta} \leq \cdot \leq \left| \frac{x_n}{\lambda_n} \right|^{\theta} }
\]
and the spatial cutoff function $\chi_n$ satisfying
\[
\chi_n(x) =
\begin{cases}
1, & |x + \frac{x_n}{\lambda_n}| \geq \frac{1}{2} \left| \frac{x_n}{\lambda_n} \right|, \\
0, & |x + \frac{x_n}{\lambda_n}| \leq \frac{1}{4} \left| \frac{x_n}{\lambda_n} \right|.
\end{cases}
\]
Additionally, $\chi_n$ obeys the bounds
\[
|\partial^\alpha \chi_n| \lesssim \left| \frac{x_n}{\lambda_n} \right|^{-|\alpha|}
\]
for all multi-indices $\alpha$, and $\chi_n \to 1$ pointwise as $n \to \infty$.

\medskip
Define the approximate solution $v_{n,T}$ as follows. Let
\[
I_{n,T}:=[a^-_{n,T},a^+_{n,T}] := [-\lambda_n^4 t_n - \lambda_n^4 T, \, -\lambda_n^4 t_n + \lambda_n^4 T]
\]
and set
\begin{align*}
    v_{n,T}(t)&=g_n[\chi_nP_ne^{i(\lambda_n^{-4}t+t_n)\Delta^2}\phi]\\
    &=\chi_n(\frac{x-x_n}{\lambda_n})e^{i(t+\lambda_n^4t_n)\Delta^2}g_n[P_n\phi].
  \end{align*}
For times $t$ outside $I_{n,T}$, extend $v_{n,T}$ linearly via
\[
v_{n,T}(t) =
\begin{cases}
e^{i (t - a_{n,T}^+) \Delta^2} v_{n,T}(a_{n,T}^+), & t > a_{n,T}^+, \\
e^{i (t - a_{n,T}^-) \Delta^2} v_{n,T}(a_{n,T}^-), & t < a_{n,T}^-.
\end{cases}
\]

\medskip
\textbf{Step 1: Uniform Bounds.} We claim
\[
\limsup_{T \to \infty} \limsup_{n \to \infty} \Big( \|v_{n,T}\|_{L_t^\infty \dot{H}^2} + \|v_{n,T}\|_{B(\mathbb{R})} \Big) \lesssim 1.
\]
This holds by standard Strichartz estimates combined with Sobolev embeddings. The key is the uniform bounds for the cutoff functions:
\[
\|\chi_n\|_{L^\infty} + \|\nabla \chi_n\|_{L^N} + \|D^2 \chi_n\|_{L^{N/2}} \lesssim 1.
\]
A standard computation shows
\begin{align*}
\|v_{n,T}\|_{L_t^\infty\dot{H}^2(\R\times\R^N)}\lesssim& \|\Delta(P_ne^{i(\lambda_n^{-4}t+t_n)\Delta^2}\phi)\|_{L^2} +\|\nabla(P_ne^{i(\lambda_n^{-4}t+t_n)\Delta^2}\phi)\|_{L^{\frac{2N}{N-2}}}\\
&+\|P_ne^{i(\lambda_n^{-4}t+t_n)\Delta^2}\phi\|_{L^{\frac{2N}{N-4}}}\lesssim1
  \end{align*}
and
$$\|v_{n,T}\|_{B(\R)}\lesssim1.$$
With the $L_t^\infty\dot{H}^2$ bound in place on $I_{n,T}$, the desired bounds $I_{n,T}^{\pm}$ follow from Sobolev embedding and Strichartz.

\medskip
\textbf{Step 2: Approximation at time zero.} We verify that
\[
\lim_{T \to \infty} \limsup_{n \to \infty} \|v_{n,T}(0) - \phi_n\|_{\dot{H}^2} = 0.
\]
If $t_n \equiv 0$, this follows directly from
\[
\|v_{n,T}(0) - \phi_n\|_{\dot{H}^2} = \| ( \chi_n P_n - 1 ) \phi \|_{\dot{H}^2} \to 0.
\]
If $t_n \to \pm \infty$, then
\[
v_{n,T}(0) = g_n e^{i t_n \Delta^2} \left[ e^{-iT\Delta^2} \chi_n P_n e^{iT\Delta^2} \phi \right],
\]
so that
\[
\|v_{n,T}(0) - \phi_n\|_{\dot{H}^2} = \| ( \chi_n P_n - 1 ) e^{iT\Delta^2} \phi \|_{\dot{H}^2} \to 0.
\]

\medskip
\textbf{Step 3: Error Control.} Define the error
\[
e_{n,T} := (i\partial_t + \Delta^2) v_{n,T} - |x|^{-b} |v_{n,T}|^p v_{n,T}.
\]
We claim that
\[
\lim_{T \to \infty} \limsup_{n \to \infty} \|e_{n,T}\|_{N(I)} = 0.
\]

We first analyze the interval $I_{n,T}$. Decompose $e_{n,T} = e_{n,T}^{\mathrm{lin}} + e_{n,T}^{\mathrm{nl}},$ where
\[
e_{n,T}^{\mathrm{lin}} = \Delta^2 \Big( \chi_n\Big( \frac{x - x_n}{\lambda_n} \Big) e^{i(t + \lambda_n^4 t_n)\Delta^2} g_n [P_n \phi] \Big)- \chi_n\Big( \frac{x - x_n}{\lambda_n} \Big) e^{i(t + \lambda_n^4 t_n)\Delta^2} \Delta^2 g_n [P_n \phi] 
\]
and
\[
e_{n,T}^{\mathrm{nl}} = \lambda_n^{-\frac{(N-4)p}{2}} g_n \Big\{ |\lambda_n x + x_n|^{-b} \chi_n^{p+1} |\Phi_n|^p \Phi_n \Big\},
\]
with $\Phi_n(t,x) = P_n e^{i(\lambda_n^{-4}t + t_n)\Delta^2} \phi$.

This produces a sum of terms of the form
\[
\partial^j \Big[\chi_n\Big( \frac{x - x_n}{\lambda_n} \Big)\Big] e^{i(t + \lambda_n^4 t_n)\Delta^2} \partial^{6-j} [g_n P_n \phi],\;\;for \;\; j \in \{1,2,3,4,5,6\}.
\]

We estimate each such term in $L_t^1 L_x^2 (I_{n,T} \times \mathbb{R}^N)$ as follows. By H\"older's inequality and Bernstein's inequality,
\begin{align*}
\Big\| \partial^j &\Big[\chi_n\Big( \frac{x - x_n}{\lambda_n} \Big)\Big] e^{i(t + \lambda_n^4 t_n)\Delta^2} \partial^{6-j} [g_n P_n \phi] \Big\|_{L_t^1 L_x^2 (I_{n,T} \times \mathbb{R}^N)} \\
&\lesssim |I_{n,T}| \Big\| \partial^j \Big[\chi_n\Big( \frac{x - x_n}{\lambda_n} \Big)\Big] \Big\|_{L_x^\infty} \Big\| \partial^{6-j} [g_n P_n \phi] \Big\|_{L_t^\infty L_x^2} \\
&\lesssim T \Big|\frac{x_n}{\lambda_n}\Big|^{-j} \Big\| \partial^{6-j} P_n \phi \Big\|_{L_t^\infty L_x^2} \\
&\lesssim T \Big|\frac{x_n}{\lambda_n}\Big|^{-j} \Big|\frac{x_n}{\lambda_n}\Big|^{|4-j|\theta} \to 0
\end{align*}
as $n \to \infty$ for $\theta$ sufficiently small.

\medskip

To estimate $e_{n,T}^{\mathrm{nl}}$ on $I_{n,T}$, we apply H\"older's inequality:
\[
\|\nabla e_{n,T}^{\mathrm{nl}}\|_{L_t^{2} L_x^{\frac{2N}{N+2}}} \leq \lambda_n^{b} T^{1/2} \Big\| \nabla \Big[ |\lambda_n x + x_n|^{-b} \chi_n^{p+1} |\Phi_n|^p \Phi_n \Big] \Big\|_{L_t^\infty L_x^{\frac{2N}{N+2}}}.
\]
Observe that
\[
\Big\| \partial^j \Big( |\lambda_n x + x_n|^{-b} \chi_n^{p-1} \Big) \Big\|_{L_x^\infty} \lesssim \Big|\frac{x_n}{\lambda_n}\Big|^{-j} |x_n|^{-b}, \quad j \in \{0,1\}.
\]
Using H\"older's inequality, Sobolev embedding, and Bernstein, we estimate
\begin{align*}
\|\nabla e_{n,T}^{\mathrm{nl}}\|_{L_t^{2} L_x^{\frac{2N}{N+2}}} \lesssim& T^{1/2} \Big|\frac{x_n}{\lambda_n}\Big|^{-b} \|\Phi_n\|_{L_t^\infty L_x^{Np}}^{p} \sum_{j=0}^1 \Big|\frac{x_n}{\lambda_n}\Big|^{-j} \|\partial^{1-j} \Phi_n\|_{L_t^\infty L_x^2} \\
\lesssim& T^{1/2} \Big|\frac{x_n}{\lambda_n}\Big|^{-b}\Big|\frac{x_n}{\lambda_n}\Big|^{\theta(3-b)}\Big|\frac{x_n}{\lambda_n}\Big|^{\theta(1+j)} \Big|\frac{x_n}{\lambda_n}\Big|^{-j} \\
\lesssim& T^{1/2} \Big|\frac{x_n}{\lambda_n}\Big|^{-b + \theta (5-b)} \to 0
\end{align*}
as $n \to \infty$ for $\theta$ sufficiently small.

\medskip

Finally, the error estimate on the intervals $I_{n,T}^\pm$ is straightforward by Sobolev embedding and linear Strichartz estimates, and is omitted.

\medskip
\textbf{Conclusion.} Applying Proposition \ref{stability}, there exist solutions $v_n$ with initial data $\Phi_n$, and these solutions have the following estimates
\begin{align}\label{Sca-bound}
  \|v_n\|_{B(\R)}\leq C\|\phi\|_{\dot{H}^2}\ \textnormal{and}\ \limsup_{T\to\infty}\lim_{n\to\infty}\|v_n-v_{n,T}\|_{B(\R)}=0.
\end{align}

With \eqref{Sca-bound} in place, we can adapt the same arguments from \cite{KMVZZ} to prove the approximation by $C_c^\infty$ functions. Here, we omit the proof.
\end{proof}

Next, we turn to the main ingredient of our work: the Palais–Smale condition. The proof is technical, combining the profile decomposition with the previous proposition to construct the associated nonlinear profiles. Two additional lemmas are crucial to complete the argument: the existence of a bad profile and the decoupling of kinetic energy.

\begin{proposition}[\textbf{Palais--Smale Condition}]\label{PS}
Let $u_n: I_n\times \R^N \rightarrow \C$ be a sequence of maximal lifespan solutions to \eqref{IBNLS} and let $t_n \in I_n$. Suppose that
\begin{equation}\label{Hyp1}
\lim_{n\rightarrow\infty} \sup_{t\in I_n} \|u_n(t)\|^2_{\dot{H}^2} = K_c < \|W\|^2_{\dot{H}^2}
\end{equation}
and
\[
\lim_{n\rightarrow \infty} \|u_n\|_{B(t\geq t_n)} = \lim_{n\rightarrow \infty} \|u_n\|_{B(t\leq t_n)} = \infty,
\]
then there exists $\{\lambda_n\} \subset \R^{+}$ such that $\{\lambda_n^{-\frac{N-4}{2}} u_n(t_n, \tfrac{x}{\lambda_n})\}$ is precompact in $\dot{H}^2$.
\end{proposition}

\begin{proof}
By the time translation invariance, we may assume that $t_n \equiv 0$, so that
\[
\lim_{n\rightarrow \infty} \|u_n\|_{S(t\geq 0)} = \lim_{n\rightarrow \infty} \|u_n\|_{S(t\leq 0)} = \infty.
\]
Since the initial data \( u_{n,0} = u_n(0) \) is uniformly bounded by~\eqref{Hyp1}, we invoke the linear profile decomposition, possibly after extracting a subsequence:
\[
u_n(0) = \sum_{j=1}^{J} g_n^j [e^{it_n^j\Delta^2} \phi^j] + w_n^J.
\]

We now turn to the construction of nonlinear profiles solving \eqref{IBNLS}. Fix $j$. If $\big|\tfrac{x_n^j}{\lambda_n^j}\big| \to \infty$ (possibly along a subsequence), then by Proposition~\ref{N-profile}, there exists a global-in-time solution $v_n^j$ to \eqref{IBNLS} with initial data given by $g_n^j [e^{it_n^j\Delta^2} \phi^j]$, and satisfying $\|v_n^j\|_{B(\R)} < \infty$. On the other hand, if $\big|\tfrac{x_n^j}{\lambda_n^j}\big|$ converges to a finite value, we may assume, without loss of generality, that $x_n^j \equiv 0$. In this setting, following the argument in~\cite{KM}, we construct the nonlinear profile $v^j$ associated with the pair $(\phi^j,{t_n^j})$ in such a way that
\begin{equation}\label{PS00}
\|v^j(t_n^j) - e^{it_n^j\Delta^2} \phi^j\|_{\dot{H}^2} \to 0 \quad \text{as } n \to \infty.
\end{equation}
In that case, define
\[
v^j_n(t, x) = \lambda_n^{-\frac{N-4}{2}} v^j\left( \tfrac{t}{(\lambda_n^j)^4} + t_n^j, \tfrac{x}{\lambda_n^j} \right).
\]
Note that $v_n^j(0) = g_n^j v^j(t_n^j)$.

Moreover, the combination of relation~\eqref{energy-decoupling} and small data theory ensures the existence of $J_0 \geq 1$ such that, for all $j \geq J_0$, the norm $\|\phi^j\|_{\dot{H}^2}$ is small and the corresponding solutions $v_n^j$ are global, satisfying
\begin{equation}\label{PSC1}
\|v_n^j\|_{L^\infty \dot{H}^2(\R)} + \|v_n^j\|_{B(\R)} \lesssim \|\phi^j\|_{\dot{H}^2}.
\end{equation}
Note that for $j < J_0$, there exists at least one bad nonlinear profile such that
\begin{equation}\label{PSC2}
\|v_n^j\|^2_{L^\infty_t \dot{H}^2(I_n^j)} \geq K_c.
\end{equation}
In fact, if every $j$, the inequality $\|v_n^j\|^2_{L^\infty_t \dot{H}^2(I_n^j)} < K_c$ holds, then by the definition of $K_c$, each $v_n^j$ would be global and $\|v_n^j\|_{B(\R)}$ would remain bounded, contradicting the lemma stated below (which will be proven afterward).

\begin{lemma}[\textbf{At least one bad profile}]\label{bp}
There exists $1 \leq j_0 < J_0$ such that
\begin{equation}\label{bp1}
\| v_n^{j_0} \|_{B([0,T_{n,j_0}))} = \infty.
\end{equation}
\end{lemma}

Now, after reordering the indices, we may assume there exists $1 \leq J_1 < J_0$ such that
\[
\limsup_{n \to \infty} \|v_n^j\|_{B([0,T_n^j))} = \infty \quad \text{for } 1 \leq j \leq J_1,
\]
and
\[
\limsup_{n \to \infty} \|v_n^j\|_{B([0,\infty))} < \infty \quad \text{for } j > J_1.
\]

For each $m, n \geq 1$, define $k_{n,m} \in \{1, \dots, J_1\}$ and interval $I_n^m = [0,s]$ by
\[
\sup_{1 \leq j \leq J_1} \|v_n^j\|_{B(I_n^m)} = \|v_n^{k_{n,m}}\|_{B(I_n^m)} = m.
\]
By the pigeonhole principle, there must exist $j_1$ such that $k_{n,m} = j_1$ for infinitely many $n$. Reordering, we assume $j_1 = 1$, hence
\[
\limsup_{n,m \to \infty} \|v_n^1\|_{B_0(I_n^m)} = \infty,
\]
which implies, using~\eqref{PSC2},
\begin{equation}\label{PSC3}
\limsup_{n,m \to \infty} \sup_{t \in I_n^m} \|\Delta v_n^1(t)\|^2_{L^2} \geq K_c.
\end{equation}

Since all $v_n^j$ have finite scattering size on $I_n^m$ for each $m \geq 1$, we obtain the same approximation:
\[
\lim_{J \to \infty} \limsup_{n \to \infty} \|u_n^J - u_n\|_{L^\infty_t \dot{H}^2_x(I_n^m)} = 0.
\]

To conclude the proof, we use the following lemma (proved later):

\begin{lemma}[\textbf{Kinetic energy decoupling for $u_n^J$}]\label{KED}
For all $J \geq 1$ and $m \geq 1$,
\[
\limsup_{n \to \infty} \sup_{t \in I_n^m} \left| \|\Delta u_n^J(t)\|^2_{L^2} - \sum_{j=1}^J \|\Delta v_n^j(t)\|^2_{L^2} - \|\Delta w_n^J\|^2_{L^2} \right| = 0.
\]
\end{lemma}

Applying the lemma above, \eqref{Hyp1} and \eqref{PSC3}, we get
\[
K_c \geq \limsup_{n \to \infty} \sup_{t \in I_n^m} \|\Delta u_n^J(t)\|^2_{L^2}
= \lim_{J \to \infty} \limsup_{n \to \infty} \left\{ \|\Delta w_n^J\|^2_{L^2} + \sup_{t \in I_n^m} \sum_{j=1}^J \|\Delta v_n^j(t)\|^2_{L^2} \right\}.
\]
Thus, relation~\eqref{PSC3} implies that $J_1 = 1$, $v_n^j \equiv 0$ for $j \geq 2$, and $w_n^1$ converges strongly to $0$ in $\dot{H}^2$, so
\[
u_n(0) = g_n^1[e^{it_n^1\Delta^2} \phi^1] + w_n^1.
\]

To show that $u_n(0)$ is precompact in $\dot{H}^2$, we must show that the space-time parameters satisfy $(t_n^1, x_n^1) \equiv (0,0)$, i.e.,
\[
u_n(0) - g_n^1(\phi^1) \to 0 \quad \text{in } \dot{H}^2.
\]

Suppose $x_n^1 \not\equiv 0$. If $|\tfrac{x_n^1}{\lambda_n^1}| \to \infty$, then by Proposition~\ref{N-profile}, there exists a global solution $v_n^1$ with
\[
v_n^1(0) = g_n^1[e^{it_n \Delta^2} \phi^1], \quad \|v_n^1\|_{B_0(\R)} < \infty.
\]
Then, by stability (Proposition~\ref{stability}), $\|u_n\|_{S(\R)} < \infty$, a contradiction.

Suppose $t_n^1 \to \infty$. By Strichartz and monotone convergence,
\[
\|e^{it\Delta^2} u_n(0)\|_{B([0,\infty))} \lesssim \|e^{it\Delta^2} \phi^1\|_{B([t_n^1,\infty))} + \|e^{it\Delta^2} w_n^1\|_{B([0,\infty))},
\]
which is small for large $n$, so by small data theory, $\|u_n\|_{B([0,\infty))}$ is bounded, contradiction. Similarly, we get a contradiction if $t_n^1 \to -\infty$.
\end{proof}

\subsection{Proofs of lemmas}
In this subsection, we show the technical lemmas used above.


\begin{proof}[\bf Proof of Lemma \ref{bp}] We show that there exists $1 \leq j_0 < J_0$ such that $\| v_n^{j_0}\|_{B([0,T_{n,j_0}))} = \infty$. Indeed, The proof follows by contradiction: if \eqref{bp1} does not hold, then $T_{n,j_0}= \infty$ and for all $1 \leq j < J_0$,
\begin{equation}\label{bp00}
\limsup_{n \rightarrow \infty} \| v_n^{j}\|_{B([0,\infty))}
< \infty.
\end{equation}
Subdividing $[0, \infty)$ into intervals where the scattering size of $v^j_n$ is small, applying the Strichartz inequality on each such interval, and then summing,
we obtain
\begin{equation}\label{bp01}
\|\Delta v_n^{j}\|_{W([0,\infty))}<\infty.
\end{equation}
Combining \eqref{PSC1}, \eqref{bp00} and \eqref{Hyp1}, we deduce (for $n$ sufficiently large)
\begin{equation*}
   \sum_{j\geq 1} \| v_n^{j}\|_{B([0,\infty))}\lesssim 1+ \sum_{j\geq J_0}\|\Delta \phi^j\|_{L^2}^2\lesssim 1+K_c.
\end{equation*}
From these assumptions, we'll establish a bound on the forward-in-time scattering size of $u_n$, thus leading to a contradiction. Indeed, define
\[
u_n^{J} = \sum_{j=1}^{J} v_n^j+e^{it\Delta^2}w_n^J.
\]
If the following conditions hold:
\begin{align}
& \limsup_{n\to\infty}\bigl\{ \|u_n^{J}\|_{L_t^\infty \dot H_x^2} + \|u_n^{J}\|_{B(\R)}\bigr\}\lesssim 1, \label{unJbds} \\
&\limsup_{n\to\infty}\|u_n(0)-u_n^{J}(0)\|_{\dot{H}^2}=0,\label{unJvan}\\
& \limsup_{n\to\infty} \|\nabla[(i\partial_t + \Delta^2)u_n^{J} - |x|^{-b}|u_n^{J}|^{p} u_n^{J}] \|_{L_t^2 L_x^{\frac{2N}{N+2}}}\lesssim \varepsilon, \label{unJapprox}
\end{align}
then the stability result (Proposition \ref{stability}) shows that $\|u_n\|_{B([0,\infty))}<\infty$, which contradicts the fact that $\|u_n\|_{B\left([0,\infty)\right)}\to \infty$ as $n\to\infty$.

We first show \eqref{unJbds}. Note that
\begin{align*}
\Big|\sum_{j=1}^Jv_n^j\Big|^2=\sum_{j=1}^J|v_n^j|^2+\sum_{j\neq k}v_n^jv_n^k,
\end{align*}
then taking the $L_{t}^{\frac{N+4}{N-4}}L_x^{\frac{N+4}{N-4}}$ norm in both side, we can get
\begin{align*}
 \Big\|(\sum_{j=1}^{J}v_n^j)^2\Big\|_{L_{t}^{\frac{N+4}{N-4}}L_x^{\frac{N+4}{N-4}}}\leq
\sum_{j=1}^{J}\|v_n^j\|_{B([0,\infty))}^2+\sum_{j\neq k}\|v_n^jv_n^k\|_{L_{t}^{\frac{N+4}{N-4}}L_x^{\frac{N+4}{N-4}}}.
\end{align*}
Since $\lim_{n\to\infty}\|v_n^jv_n^k\|_{L_{t}^{\frac{N+4}{N-4}}L_x^{\frac{N+4}{N-4}}}=0$ (by the orthogonality conditions \eqref{orthogonality}) we have
\begin{align*}
\limsup_{n\to\infty}\|u_n^J\|_{B([0,\infty)}\leq \limsup_{n\to\infty}\sum_{j=1}^J\|v_n^{j}\|_{B([0,\infty))}\lesssim 1+K_c,
\end{align*}
which is independent of $J$. Moreover, $\|e^{it\Delta^2}w_n^J\|_{B\left([0,\infty)\right)}\lesssim \|w_n^J\|_{\dot{H}^2}$ is finite. Therefore, the first estimate holds. In the same way (using \eqref{bp01}), we also obtain that $\|u_n^J\|_{L^\infty \dot{H}^2}$ is bounded.

We now consider \eqref{unJvan}. Using the fact that $v_n^j(0)=g_n^jv^j(t_n^j)$
$$
\|u_{0,n}-u_{n}^{J}(0)\|_{\dot{H}^1}\lesssim\sum_{j=1}^J \left\|\left( g_n^j[e^{it^j_n\Delta^2}\phi^j]-g_n^jv^j(t_n^j)\right)\right\|_{\dot{H}^2}\lesssim \sum_{j=1}^J \| e^{it^j_n\Delta^2}\phi^j-v^j(t_n^j)\|_{\dot{H}^2},
$$
which goes to zero, as $n\rightarrow \infty$, by \eqref{PS00}.

\ Finally, we turn to show \eqref{unJapprox}. To this end, we write $f(z)=|z|^{p} z$ and observe
\begin{align}\label{enJ1}
e_n=(i\partial_t + \Delta)u_n^J -|x|^{-b} f(u_n^J)=&|x|^{-b}\left[f\bigl(\sum v_n^j\bigr) - \sum f(v_n^j)\right].
\end{align}
Given \eqref{SECONDEI}, to estimate $\nabla e_n$, it is sufficient to estimate terms of the following types:
\begin{itemize}
\item[(1)] $ |x|^{-b} |v_n^j|^{p-1}\cdot v_n^\ell \cdot T v_n^k $,\quad $T\in\{|x|^{-1},\nabla\}$,
\item[(2)] $ |x|^{-b} |e^{it\Delta^2}w_n^{J}|^{p} \cdot T e^{it\Delta^2}w_n^{J} $,\quad $T\in\{|x|^{-1},\nabla\}$,
\item[(3)] $ |x|^{-b} |e^{it\Delta^2}w_n^{J}|^{p} \cdot \nabla u_n^{J}$,
\item[(4)] $ |x|^{-b} |u_n^{J}|^{p-1}\cdot e^{it\Delta^2}w_n^{J} \cdot \nabla u_n^{J}$,
\item[(5)] $ |x|^{-b} |u_n^{J}|^{p} \cdot T e^{it\Delta^2}w_n^{J} $,\quad $T\in\{|x|^{-1},\nabla\}$,
\end{itemize}
where we have $j\neq k$ and $\ell\in\{1,\dots,J\}$.  The first term also involves a constant $C_{J}$ that grows with $J$.  However, we will shortly see that for each fixed $J$, these terms tend to zero in $L_t^2 L_x^{\frac{2N}{N+2}}$ as $n\to\infty$, so that this constant is ultimately harmless.

Let $q_0, r_0, \bar{r}, \beta$ denote by
\begin{align*}
  q_0=\frac{2(N+4)(b+1)}{b(N-2)+N-4},\ r_0=\frac{2N(N+4)(b+1)}{N^2+b(N^2+8)+16}\ \text{and}\ \beta=\frac{Nr_0}{N-r_0}.
\end{align*}
By Lemma~\ref{Non-e1}, it follows that

\begin{align*}
 \|(1)\|_{L_t^2 L_x^{\frac{2N}{N+2}}}\lesssim& \||x|^{-1}v_n^j\|_{L_t^{q_0}L_x^{\beta}}^b\||v_n^j|^{p-b-1}Tv_n^k\|_{L_t^{q_1}L_x^{r_1}}\|v_n^l\|_{L_t^{\frac{2(N+4)}{N-4}}L_x^{\frac{2(N+4)}{N-4}}}\\
 \lesssim&\|\Delta v_n^j\|_{L_t^{q_0}L_x^{r_0}}\||v_n^j|^{p-b-1}Tv_n^k\|_{L_t^{q_1}L_x^{r_1}}\|v_n^l\|_{B([0,\infty))},
\end{align*}
where $(q_1,r_1)$ satisfies
$$\frac{1}{q_1}=\frac1{q_0}+\frac{(N-4)(p-b-1)}{2(N+4)},\ \ \frac{1}{r_1}=\frac1{\beta}+\frac{(N-4)(p-b-1)}{2(N+4)}$$
and
$$1=\frac{N}{r_0}-\frac{N}{\beta}.$$
As the norms $\|v_n^\ell\|_{B([0,\infty))}$ and $\|\nabla v_n^j\|_{L_t^{q_0}L_x^{r_0}}$ are bounded. It suffices to show that
\begin{align}\label{np-decouple}
  \lim_{n\to\infty}\||v_n^k|^{p-b-1}Tv_n^j\|_{L_t^{q_1}L_x^{r_1}}=0\ \ \text{for}\ \ j\neq k.
\end{align}
In fact, it follows from approximation by functions in $C_c^\infty(\R\times\R^N)$ and the use of the orthogonality conditions \eqref{orthogonality} (see e.g. \cite{Ker} or \cite[Lemma~7.3]{Visan}).

 Next, we turn to estimate terms $(2),(3), (4)$. Choosing $J>J'$ large enough such that
 $$\lim_{n\to\infty}\|e^{it\Delta^2}w_n^{J}\|_{B\left([0,\infty)\right)}<\frac{\varepsilon}{J'}.$$
 By the lemma \ref{Non-e1}, we estimate
 \begin{align*}
\|(2)\|_{L_t^2 L_x^{\frac{2N}{N+2}}}\lesssim
\|e^{it\Delta^2}w_n^{J}\|_{B\left([0,\infty)\right)}^{p-b}\|\Delta e^{it\Delta^2}w_n^{J}\|_{W([0,\infty))}^{1+b},
\end{align*}
 $$
 \|(3)\|_{L_t^2 L_x^{\frac{2N}{N+2}}}\lesssim \|e^{it\Delta^2}w_n^{J}\|_{B\left([0,\infty)\right)}^{p-b}\|\Delta u_n^{J}\|_{W([0,\infty)}\|\nabla e^{it\Delta^2}w_n^{J}\|_{W([0,\infty))}^{b}
 $$
 and
  $$
 \|(4)\|_{L_t^2 L_x^{\frac{2N}{N+2}}}\lesssim \|u_n^J\|_{B([0,\infty))}^{p-b-1} \|e^{it\Delta^2}w_n^{J}\|_{B\left([0,\infty)\right)}\|\Delta u_n^{J}\|_{W([0,\infty)}^{b+1}.
 $$
Together with vanishing condition \eqref{vanishing} and $\|w_n^{J}\|_{\dot H^1}\lesssim1$, we have
$$\limsup_{n\to\infty}\big(\|(2)\|_{L_t^2 L_x^{\frac{2N}{N+2}}}+\|(3)\|_{L_t^2 L_x^{\frac{2N}{N+2}}}+\|(4)\|_{L_t^2 L_x^{\frac{2N}{N+2}}}\big)\leq \varepsilon.$$

Now we turn to estimate $(5)$.  We only estimate a single term of the form, for $j\in\{1,...,J\}$,
$$\||x|^{-b} |v_n^{j}|^{p}\nabla e^{it\Delta^2}w_n^{J} \|_{L_t^2L_x^{\frac{2N}{N+2}}}.$$

By density, we may assume $v^{j}\in C_c^\infty(\R\times\R^N\setminus\{0\})$. Applying H\"older's inequality, the problem further reduces to showing that
\begin{align}\label{YU}
  \lim_{J\to J^*}\limsup_{n\to\infty}\|\nabla \tilde{w}_n^{J}\|_{L_{t,x}^2(K)}=0\ \text{for any compact $K\subset \R\times \R^N$},
\end{align}
where $\tilde{w}_n^{J}:=(\lambda_n^j)^{\frac{N-4}2}e^{i((\lambda_n^j)^4t-t_n^j)\Delta^2}w_n^{J}(\lambda_n^jx+x_n^j).$
and we have
$$\|\tilde{w}_n^{J}\|_{B(\R)}=\|e^{it\Delta^2}w_n^J\|_{B(\R)}.$$
This finally follows from an interpolation argument using  Lemma \ref{Local-smoothing} and vanishing \eqref{vanishing}(See \cite[Lemma 2.12]{KV} for details).
\end{proof}

We conclude this subsection by proving the decoupling of the kinetic energy.

\begin{proof}[The proof of Lemma \ref{KED}] The proof follows closely the argument of Lemma 3.2 in \cite{KV}; thus, we outline only the main steps. For any \( t_n \in I_n^m \), a direct computation yields
\[
 \|\Delta u_n^J(t)\|_{L^2}^2 - \sum_{j=1}^J \|\Delta v^j(t) \|_{L^2}^2 - \|\Delta w_n^J\|_{L^2}^2 = I_n + II_n,
\]
where
\[
I_n = \sum_{j \neq k} \langle v_n^j(t_n^j), v_n^k(t_n^k) \rangle_{\dot{H}^2}, \qquad \text{and} \qquad II_n = \sum_{j=1}^J 2 \Re \langle e^{it \Delta} w_n^J, v_n^j(t_n^j) \rangle_{\dot{H}^2}.
\]
It is therefore sufficient to prove that
\[
\lim_{n \to \infty} I_n = 0 \quad \text{for } j \neq k, \qquad \text{and} \qquad \lim_{n \to \infty} II_n = 0.
\]
To handle the term \( I_n \), we invoke the fact that for all \( J \geq 1 \) and \( 1 \leq j \leq J \), the sequence \( e^{it_n^j \Delta^2} \left[(g_n^j)^{-1} w_n^J \right] \) converges weakly to zero in \( \dot{H}^2 \) as \( n \to \infty \). The estimate for \( II_n \) proceeds analogously, relying on the orthogonality of parameters given in \eqref{orthogonality}.
\end{proof}


With the key analytical tools in place, most notably the Palais–Smale condition, we are now in a position to construct the critical solution to \eqref{IBNLS}. That is, if the main theorem fails, there exists a critical solution \( u_c \), with energy below that of the ground state \( W \), which does not scatter and whose rescaled orbit is precompact in \( \dot{H}^2 \). This solution represents the minimal obstruction to global scattering.

\begin{proposition}{(\bf Critical solution)}\label{CS}
Assume Theorem \ref{main-T} fails, then there exist a critical value $0 < K_c < \|W\|^2_{\dot{H}^2}$ and a forward maximal life-span solution $u_c: [0, T_{\max}) \times \R^N
\rightarrow \mathbb{C}$ to \eqref{IBNLS} with

$$
\sup_{t\in [0, T_{\max})}\|u_c(t)\|^2_{\dot{H}^2}=K_c<\|W\|^2_{\dot{H}^2}\quad \textnormal{and}\quad \|u_c\|_{B_0\left([0, T_{\max})\right)}=\infty.
$$
Moreover, there exists a frequency scale function $\lambda : [0, T_{\max}) \rightarrow (0, \infty)$ such that $$K=\{ \lambda(t)^{-\frac{N-4}{2} } u_c(t, \lambda(t)^{-1}x) : t \in [0, T_{\max})\}$$
is precompact in $\dot{H}^2$. An analogous result holds backward in time.

\end{proposition}
\begin{proof}
The proof primarily relies on the Palais–Smale condition (for further details, see \cite{KM}, \cite{KV}). We outline the main steps. Suppose Theorem \ref{main-T} fails. By the definition of \( K_c \), there exists a sequence of solutions \( u_n : I_n \times \R^N \rightarrow \mathbb{C} \) to \eqref{IBNLS} such that
\begin{equation*}
\lim_{n\rightarrow\infty}\sup_{t\in I_n}\|u_n(t)\|^2_{\dot{H}^2}= K_c < \|W\|^2_{\dot{H}^2}\,,\qquad \lim_{n\rightarrow \infty} \|u_n\|_{B(t\geq t_n)} = \lim_{n\rightarrow \infty} \|u_n\|_{B(t\leq t_n)} = \infty.
\end{equation*}

By choosing \( t_n \in I_n \) such that \( \| u_n(t) \|_{B(t \geq t_n)} = \| u_n(t) \|_{B(t \leq t_n)} \) and using the time-translation invariance of the equation, we set \( t_n \equiv 0 \). Then, applying Proposition \ref{PS} and passing to a subsequence if necessary, there exists a sequence \( \{ \lambda_n \} \subset \mathbb{R}^+ \) such that
\[
\lambda_n^{-\frac{N-4}{2}} u_n(0, x / \lambda_n) \to u_0 \quad \text{in } \dot{H}^2.
\]

Let \( u_c : I_{\max} \times \mathbb{R}^N \to \mathbb{C} \) with \( u_c(0) = u_0 \) be the maximal-lifespan solution to \eqref{IBNLS}. By the stability theory, for any compact interval \( I \subset I_{\max} \), we have
\[
\begin{cases}
\lim_{n \rightarrow \infty}\| u_n - u_c \|_{L^\infty_t \dot{H}^2_x(I \times \mathbb{R}^N)} = 0, \\
\| u_c(t) \|_{B\left([0,\infty)\right)} = \| u_c(t) \|_{B\left((-\infty,0]\right)} = \infty, \\
\sup_{t \in I} \| u_c(t) \|_{\dot{H}^2} = K_c < \| W \|^2_{\dot{H}^2}.
\end{cases}
\]

Moreover, invoking Proposition \ref{PS} once more, we conclude that the trajectory \( K \) is precompact in \( \dot{H}^2 \).
\end{proof}



\section{Proof of main result: Theorem \ref{main-T}}

We argue by contradiction. Assume that Theorem 1.2 fails. Then, by Proposition \ref{CS}, there exists a minimal blow-up solution $
u_c : [0, T_{\max}) \times \mathbb{R}^N \to \mathbb{C},$ with the compactness property. Our goal is to rule out the existence of such a non-scattering critical solution. We divide the analysis into two cases according to the maximal time of existence.

\medskip
\noindent
\textbf{Claim 1.} There is no solution of the form given in Proposition 3.7 with \( T_{\max} < \infty \).

\smallskip
To prove this, we invoke the reduced Duhamel formula, which is a consequence of the compactness of \( u_c \) (see, for example, \cite[Proposition 5.23]{KM}):
\begin{lemma}[Reduced Duhamel formula]
\label{duhamel}
 For $t\in [0,T_{\max})$, the following holds as a weak limit in $\dot{H}^2$:
  $$u(t)=i\lim_{T\to T_{\max}}\int_{t}^{T}e^{i(t-s)\Delta^2}|x|^{-b}|u|^{p}u(s)ds.$$
\end{lemma}

We suppose that $T_{\max}<\infty$. Using Strichartz, Hardy--Littlewood--Sobolev's inequality, Bernstein's inequality, and Hardy's inequality, we have
  \begin{align*}
    \|P_{N}u(t)\|_{L_x^2}&\lesssim \|P_N[|x|^{-b}|u|^{p}u]\|_{L_t^1L_x^2}\\
    &\lesssim N^{N(\frac{N+4}{2N}-\frac12)}|T_{\max}-t|\||x|^{-b}|u|^{p}u\|_{L_t^\infty L_x^{\frac{2N}{N+4}}}\\
    &\lesssim N^2|T_{\max}-t|\|\Delta u\|_{L_t^\infty L_x^2}^{p+1}.
  \end{align*}
  Thus, using Bernstein's inequality again, we deduce
  \begin{align*}
    \|u(t)\|_{L_x^2}&\leq\|P_Nu(t)\|_{L_x^2}+\|(1-P_N)u(t)\|_{L_x^2}\\
   &\lesssim  N|T_{\max}-t|+N^{-2}.
  \end{align*}
  Using mass conservation, we obtain \(\|u\|_{L^2} = 0\), and thus \(u \equiv 0\). However, since \(u \not\equiv 0\), this yields a contradiction, thereby ruling out the finite-time blow-up scenario.

\medskip
\noindent
\textbf{Claim 2.} There is no solution of the form given in Proposition \ref{CS} with \( T_{\max} = \infty \).

\smallskip
To exclude this case, we consider the localized virial/Morawetz-type functional. Let $a:\R^N\rightarrow\R$ be a smooth weight. Define
$$M_a(t)=-{\rm Im}\int_{\R^N}\bar{u}\nabla u\cdot\nabla adx.$$
A direct computation using equation~\eqref{IBNLS} and integration by parts yields (see \cite{GC2022})
\begin{align*}
 \frac{d}{dt}M_a(t)=&4\sum_{i,j,k}\int_{\R^N}\partial_{jk}u\partial_{ik}\bar{u}\partial_{ij}adx-
 \int_{\R^N}\Big(\frac{p}{p+2}\Delta a+\frac{2b}{p+2}\frac{x\cdot\nabla a}{|x|^2}\Big)|x|^{-b}|u|^{p+2}dx\\
 &-2\sum_{j,k}\int_{\R^N}\partial_{jk}\Delta a\partial_j\bar{u}\partial_kudx+\frac12\int_{\R^N}\Delta^3a|u|^2dx-\int_{\R^N}\Delta^2a|\nabla u|^2dx.
\end{align*}
where subscripts denote partial derivatives and repeated indices are summed.

The standard virial identity uses \( a(x) = |x|^2 \), but with \( \dot{H}^2 \) data, \( M_a(t) \) may be infinite. Hence, it's necessary to localize the weight, which is justified by the compactness of \( u_c(t) \). This leads to the following result (the proof is standard).

\begin{lemma}[Tightness]
  Let $\epsilon>0$ and $p\geq b$.  Then there exists $R=R(\epsilon)$ sufficiently large so that
  \begin{align*}
    \sup_{t\in[0,\infty)}\int_{|x|>R}|\Delta u(t,x)|^2+|x|^{-4}|u(t,x)|^2+|x|^{-b}|u|^{p+2}(t,x)dx<\epsilon.
  \end{align*}
\end{lemma}

Returning to the proof of Claim $2$, let weight $a(x)$ denote by
\begin{align*}
a(x)=
  \begin{cases}
    |x|^2\ \ for\ \ |x|\leq R\\
    CR^2\ \ for\ \ |x|>2R
  \end{cases}
\end{align*}
for some $C>1$. In this intermediate region, we can impose
$$|D^j a|\lesssim R^{2-j}\ \ for\ \ R<|x|\leq 2R.$$
On the other hand, we have
  \begin{align}
    \frac{d}{dt}M_a(t)=&8\int_{\R^N}|\Delta u|^2-|x|^{-b}|u|^{p+2}dx\\
    &+\mathcal{O}\left(\int_{|x|>R}|x|^{-2}|\nabla u|^2+|x|^{-4}|u|^2+|x|^{-b}|u|^{p+2}dx\right),
  \end{align}
Lemma \ref{energy-tra} yields
$$E(u)\lesssim \frac{R^4}{\delta T}E(u)+\epsilon.$$
Choosing $T$ sufficiently large, we obtain $E(u)\lesssim \epsilon$. Then we deduce $E(u)\equiv0$, which implies $u\equiv0$, a contradiction.

Therefore, both scenarios \( T_{\max} < \infty \) and \( T_{\max} = \infty \) are excluded. This completes the proof of Theorem \ref{main-T}.

 \section*{Declarations}
\noindent {\bf Ethical Approval:} Not applicable.\\
{\bf Funding:} C.M.G. was partially supported by Conselho Nacional de Desenvolvimento Científico e Tecnologico - CNPq and Fundação de Amparo à Pesquisa do Estado do Rio de Janeiro - FAPERJ (Brazil). S. Keraani is supported in part by the Labex CEMPI (ANR-11-LABX-0007-01). C. Xu was partially supported by Natural Science Foundation of Qinghai
(No.2024-ZJ-976) and NSFC (No.12401296).\\
{\bf Authors' contributions:} The authors contributed equally to the preparation of this paper.\\
{\bf Availability of data and materials:} No new data or materials have been used in preparing this paper.

\bibliographystyle{abbrv}
\bibliography{bibpaper}

 \end{document}